\documentclass[a4paper,11pt]{article}
\usepackage[latin1]{inputenc}
\usepackage[english]{babel}
\usepackage{amsmath}
\usepackage{amsfonts}
\usepackage{amssymb}
\usepackage{epsfig}
\usepackage{amsopn}
\usepackage{amsthm}
\usepackage{color}
\usepackage{graphicx}
\usepackage{enumerate}
\usepackage{mathrsfs}
\usepackage{cite}
\newtheorem{theorem}{Theorem}[section]
\newtheorem{corollary}[theorem]{Corollary}
\newtheorem{lemma}[theorem]{Lemma}

\newtheorem{definition}[theorem]{Definition}

\newtheorem*{theorem*}{Theorem}
\newtheorem*{lemma*}{Lemma}
\newtheorem*{remark*}{Remark}
\newtheorem*{definition*}{Definition}
\newtheorem*{proposition*}{Proposition}
\newtheorem*{corollary*}{Corollary}
\numberwithin{equation}{section}
\newcommand{\real}{\mathbb{R}}

\let\ced=\c         

\def\qed{\,\unskip\kern 6pt \penalty 500
\raise -2pt\hbox{\vrule \vbox to8pt{\hrule width 6pt
\vfill\hrule}\vrule}\par}
\definecolor{darkblue}{rgb}{0.05, .05, .65}
\definecolor{darkgreen}{rgb}{0.1, .65, .1}
\definecolor{darkred}{rgb}{0.8,0,0}
\newcommand{\beqn}{\begin{equation}}
\newcommand{\eeqn}{\end{equation}}
\newcommand{\bear}{\begin{eqnarray}}
\newcommand{\eear}{\end{eqnarray}}
\newcommand{\bean}{\begin{eqnarray*}}
\newcommand{\eean}{\end{eqnarray*}}
\begin{document}

\title{\huge \bf A porous medium equation with spatially inhomogeneous absorption. Part II: Large time behavior}
\author{
\Large Razvan Gabriel Iagar\,\footnote{Departamento de Matem\'{a}tica
Aplicada, Ciencia e Ingenieria de los Materiales y Tecnologia
Electr\'onica, Universidad Rey Juan Carlos, M\'{o}stoles,
28933, Madrid, Spain, \textit{e-mail:} razvan.iagar@urjc.es}
\\[4pt] \Large Diana-Rodica Munteanu\,\footnote{Faculty of Psychology and Educational Sciences, Ovidius University of Constanta, 900527, Constanta, Romania, \textit{e-mail:} diana.rodica.merlusca@gmail.com}\\
}
\date{}
\maketitle

\begin{abstract}
We study the large time behavior of solutions to the Cauchy problem for the quasilinear absorption-diffusion equation
$$
\partial_tu=\Delta u^m-|x|^{\sigma}u^p, \quad (x,t)\in\real^N\times(0,\infty),
$$
with exponents $p>m>1$ and $\sigma>0$ and with initial conditions either satisfying
$$
u_0\in L^{\infty}(\real^N)\cap C(\real^N), \quad \lim\limits_{|x|\to\infty}|x|^{\theta}u_0(x)=A\in(0,\infty)
$$
for some $\theta\geq0$. A number of different asymptotic profiles are identified, and uniform convergence on time-expanding sets towards them is established, according to the position of both $p$ and $\theta$ with respect to the following critical exponents
$$
p_F(\sigma)=m+\frac{\sigma+2}{N}, \quad \theta_*=\frac{\sigma+2}{p-m}, \quad \theta^*=N.
$$
More precisely, solutions in radially symmetric self-similar form decaying as $|x|\to\infty$ with the rates
$$
u(x,t)\sim A|x|^{-\theta_*}, \quad {\rm or} \quad u(x,t)\sim \left(\frac{1}{p-1}\right)^{1/(p-1)}|x|^{-\sigma/(p-1)},
$$
are obtained as asymptotic profiles in some of these cases, while asymptotic simplifications or logarithmic corrections in the time scales also appear in other cases. The uniqueness of some of these self-similar solutions, left aside in the first part of this work, is also established.
\end{abstract}

\smallskip


\smallskip


\section{Introduction}

The goal of this paper is to perform a detailed study of the large time behavior of solutions to the Cauchy problem associated to the following absorption-diffusion equation featuring a spatially inhomogeneous absorption
\begin{equation}\label{eq1}
u_t=\Delta u^m-|x|^{\sigma}u^p, \quad (x,t)\in\real^N\times(0,\infty),
\end{equation}
with continuous, non-negative initial condition
\begin{equation}\label{ic}
u(x,0)=u_0(x)\geq0, \quad x\in\real^N, \quad u_0\in L^{\infty}(\real^N)\cap C(\real^N), \quad u_0\not\equiv0,
\end{equation}
and in the range of exponents
\begin{equation}\label{range.exp}
1<m<p, \quad \sigma>0.
\end{equation}
Eq. \eqref{eq1} features a competition between the diffusion and the absorption terms. Since the diffusion equation with $m>1$ taken alone is a conservative one, preserving the $L^1$ norm of the initial condition (in case of an integrable one) along the evolution, while the absorption tends to reduce the total mass of any solution, the effect of joining these two terms produces interesting mathematical properties and new behaviors which are not specific to any of the two terms let alone.

The homogeneous equation, that is
\begin{equation}\label{eq1.hom}
u_t=\Delta u^m-u^p
\end{equation}
is by now well understood in our range of exponents \eqref{range.exp} with respect to the asymptotic behavior of its solutions as $t\to\infty$, but its mathematical study originated several techniques that have been then employed in a large number of different problems. It has been noticed that the solutions to Eq. \eqref{eq1.hom} approach as $t\to\infty$ profiles which vary depending on

$\bullet$ the position of $p$ with respect to $p_F=m+2/N$, which has been initially identified in relation with reaction-diffusion equations (see for example \cite{Fu66, GKMS80}) but it has been shown to play a fundamental role also for Eq. \eqref{eq1.hom}.

$\bullet$ the behavior of the initial condition $u_0$ as $|x|\to\infty$.

A classification of a number of cases for the large time behavior depending on the previous thresholds has been established in \cite{KP85, KP86, KU87, KV88}. Among other results, it has been noticed that, in the range $p>p_F$, an \emph{asymptotic simplification} is in force, leading to convergence towards self-similar profiles related to the pure porous medium equation. Considering $p=p_F$ in \cite{GV91} led to the development of a new dynamical systems technique known as \emph{the S-theorem}. The most difficult range related to Eq. \eqref{eq1.hom} is the one of compactly supported (or rapidly decaying at infinity) initial conditions when $p\in(1,p_F)$. For such data, asymptotic profiles in the form of \emph{very singular self-similar solutions} had been identified in \cite{BPT86, PT86, Le96, Le97, PW88, PZ91}, to name but a few of them (including some devoted to the fast diffusion range $m<1$). The panorama of the large time behavior has been completed with the works by Kwak \cite{Kwak98, Kwak98b} establishing the asymptotic profile for a critical behavior as $|x|\to\infty$. Techniques developed in all the above mentioned works became then valuable tools for a number of different problems involving convection terms, $p$-Laplacian diffusion, gradient terms etc. Let us stress that all this research listed here depends strongly on the relation $p>m$; indeed, when $1<p\leq m$, the large time behavior is completely different, as seen for example in \cite{MPV91, CV96, CV99, CVW97} and a number of interesting problems (in higher space dimensions) are still open.

Equations such as \eqref{eq1}, involving a weighted absorption term, have been proposed in problems from mathematical biology \cite{GMC77, N80}. Motivated by these models, Peletier and Tesei studied in \cite{PT85, PT86b} the problem of expansion of supports for compactly supported initial conditions. More recently, the mathematical theory of Eq. \eqref{eq1} has been developed in the \emph{strong absorption range} $p\in(0,1)$ in connection with the phenomenon of finite time extinction in \cite{IL23, ILS24}. Since in that range the absorption is the dominating term, the mathematical properties of the solutions to Eq. \eqref{eq1} with $p\in(0,1)$ strongly differ from the ones we shall prove in the present work.

In the first part of this work \cite{IM25}, the authors made a first step towards the large time behavior of solutions to Eq. \eqref{eq1} in the range of exponents \eqref{range.exp} by classifying the self-similar solutions decaying as $t\to\infty$, in the form
\begin{equation}\label{SSS}
u(x,t)=t^{-\alpha}f(\xi), \quad \xi=|x|t^{-\beta}, \quad (x,t)\in\real^N\times(0,\infty),
\end{equation}
with
\begin{equation}\label{SSexp}
\alpha=\frac{\sigma+2}{L}>0, \quad \beta=\frac{p-m}{L}>0, \quad L=\sigma(m-1)+2(p-1).
\end{equation}
Indeed, a complete study of the initial value problem for the differential equation satisfied by the profiles $f$ of solutions in the form \eqref{SSS}, that is,
\begin{equation}\label{SSODE}
\left\{\begin{array}{ll}(f^m)''(\xi)+\frac{N-1}{\xi}(f^m)'(\xi)+\alpha f(\xi)+\beta\xi f'(\xi)-\xi^{\sigma}f^p(\xi)=0, \\
f(0)=A>0, \quad f'(0)=0, \end{array}\right.
\end{equation}
led to a list of different types of solutions based on the behavior of $f(\xi)$ as $\xi\to\infty$. It has been noticed that the exponent
\begin{equation}\label{pcrit}
p_F(\sigma)=m+\frac{\sigma+2}{N}
\end{equation}
is critical and, to keep the presentation as brief as possible, we have established the existence of three different types of self-similar profiles:

$\bullet$ a unique compactly supported self-similar profile, existing only for $m<p<p_F(\sigma)$ and being, in fact, a very singular solution. This means that its initial trace is more concentrated in the origin than a Dirac distribution; more precisely, it satisfies the following two conditions:
\begin{equation}\label{VSS}
\lim\limits_{t\to0}\int_{|x|>\epsilon}u(x,t)\,dx=0, \quad \lim\limits_{t\to0}\int_{|x|<\epsilon}u(x,t)\,dx=+\infty,
\end{equation}
for any $\epsilon>0$.

$\bullet$ a number of profiles presenting the following behavior
\begin{equation}\label{dec.slow}
\lim\limits_{\xi\to\infty}\xi^{(\sigma+2)/(p-m)}f(\xi)=K\in(0,\infty).
\end{equation}

$\bullet$ at least a profile presenting the following behavior
\begin{equation}\label{dec.zero}
\lim\limits_{\xi\to\infty}\xi^{\sigma/(p-1)}f(\xi)=K_p:=\left(\frac{1}{p-1}\right)^{1/(p-1)},
\end{equation}
and we have conjectured in \cite{IM25} that this profile is unique. Notice that this behavior strongly depends on $\sigma$ and its existence is determined by the presence of the weight $|x|^{\sigma}$ in the formulation of Eq. \eqref{eq1}.

\medskip

We end up this introduction by defining the notion of solution that will be used throughout the paper.
\begin{definition}\label{def.sol}
We say that $u$ is a weak solution to Eq. \eqref{eq1} if $u\in L^{\infty}(\real^N\times(\tau,\infty))$ for any $\tau>0$ and, for any compactly supported test function $\zeta\in C_{c}^{2,1}(\real^N\times(0,\infty))$, we have, for any $\tau>0$,
\begin{equation}\label{weak.sol}
\int_{\tau}^{\infty}\int_{\real^N}\left(\zeta_tu+\Delta\zeta u^m-|x|^{\sigma}\zeta u^p\right)\,dx\,dt+\int_{\real^N}\zeta(x,\tau)u(x,\tau)\,dx=0.
\end{equation}
As usual, a weak subsolution (respectively supersolution) is defined by restricting the class of test functions to non-negative ones and replacing the equality sign in \eqref{weak.sol} by the sign $\geq$ (respectively $\leq$). We further say that $u$ is a weak solution to the Cauchy problem \eqref{eq1}-\eqref{ic} if $u(x,t)\to u_0(x)$ as $t\to0$ in weak sense, that is, 
$$
\lim\limits_{t\to0}\int_{\real^N}u(x,t)\varphi(x)\,dx=\int_{\real^N}u_0(x)\varphi(x)\,dx, 
$$
for any $\varphi\in C_c^{1}(\real^N)$.
\end{definition}
Notice that the first part of Definition \ref{def.sol} is adapted in order to allow singular initial conditions producing bounded functions at any $t>0$. Let us remark here that any weak solution to Eq. \eqref{eq1} in the sense of Definition \ref{def.sol} satisfies $u\in C(\real^N\times(0,\infty))$, according to \cite{DiB83}. It is now the right moment to introduce the main results of this work.

\section{Main results}\label{sec.main}

Our main goal, as explained in the Introduction, is to establish the large time behavior of solutions to the Cauchy problem \eqref{eq1}-\eqref{ic} in dependence on the properties of the initial condition $u_0$, showing in particular the relevance of the above listed profiles for it. We also prove along the way the conjecture stated in \cite{IM25} related to the uniqueness of the profile behaving as in \eqref{dec.zero}, as well as some uniqueness results for profiles with the behavior given by \eqref{dec.slow}. In order to fix the notation, let us denote by $f(\cdot;A)$ and $U(\cdot,\cdot;A)$ the profile solving the initial value problem \eqref{SSODE} and, respectively, the self-similar solution defined by
\begin{equation}\label{SSSbis}
U(x,t;A)=t^{-\alpha}f(|x|t^{-\beta};A),
\end{equation}
for $A\in(0,\infty)$. For the easiness of the reading, we divide the presentation of our results into several parts.

\medskip

\noindent \textbf{A. Uniqueness results.} Our first result concerns the uniqueness of some self-similar profiles, completing thus the analysis performed in \cite{IM25}.
\begin{theorem}\label{th.uniq}
(a) There exists a unique $A^*\in(0,\infty)$ such that the self-similar profile $f(\cdot;A^*)$ solving \eqref{SSODE} with $f(0)=A^*$ presents the decay \eqref{dec.zero} as $\xi\to\infty$.

(b) For any $K\in(0,\infty)$, there exists a unique $A(K)\in(0,A^*)$ (depending on $K$) such that the self-similar profile $f(\cdot;A(K))$ solving \eqref{SSODE} with $f(0)=A(K)$ presents the limit behavior \eqref{dec.slow} as $\xi\to\infty$.
\end{theorem}

The proof of Theorem \ref{th.uniq} is based on an analysis of the differential equation \eqref{SSODE}, employing a rescaling and sliding technique. Let us mention here that, for $\sigma=0$, that is, Eq. \eqref{eq1.hom}, the uniqueness in part (a) of Theorem \ref{th.uniq} is completely obvious, since for $\sigma=0$ the corresponding solution is an explicit constant. However, the profile $f(\cdot; A^*)$ is no longer explicit and its uniqueness is no longer obvious in the case of Eq. \eqref{eq1}. With respect to part (b) in Theorem \ref{th.uniq}, the corresponding result for $\sigma=0$ is established in \cite{Kwak98}.

\medskip

\noindent \textbf{B. Large time behavior for slowly decaying initial conditions.} Taking into account the classification of the solutions to the Cauchy problem \eqref{SSODE}, it is apparent that the decay behavior \eqref{dec.slow} is critical for the general dynamical properties of Eq. \eqref{eq1}. This is why, we gather in this paragraph the statements of the results concerning the large time behavior of solutions to the Cauchy problem \eqref{eq1}-\eqref{ic} stemming from initial conditions such that
$$
\lim\limits_{|x|\to\infty}|x|^{(\sigma+2)/(p-m)}u_0(x)>0.
$$
The first statement concerns initial conditions such that the previous limit is equal to $\infty$. More precisely, the following large time behavior result holds true:
\begin{theorem}\label{th.asympt.slow1}
Let $u$ be the solution to the Cauchy problem \eqref{eq1}-\eqref{ic} with an initial condition such that
\begin{equation}\label{eq.veryslow}
\lim\limits_{|x|\to\infty}|x|^{(\sigma+2)/(p-m)}u_0(x)=\infty.
\end{equation}
Then, we have
\begin{equation}\label{asympt.slow1}
\lim\limits_{t\to\infty}t^{\alpha}|u(x,t)-U(x,t;A^*)|=0,
\end{equation}
with uniform convergence on sets of the form
\begin{equation}\label{inner.sets}
\mathcal{S}_{c}=\{x\in\real^N: |x|\leq ct^{\beta}\}, \quad c>0,
\end{equation}
where $U(x,t;A^*)$ is the solution defined in \eqref{SSSbis} corresponding to the unique value $A^*$ introduced in Theorem \ref{th.uniq}.
\end{theorem}
We thus observe that a large amount of initial conditions that either present a horizontal asymptote or decay to zero as $|x|\to\infty$ but with a rather slow decay rate give rise to solutions that are attracted, as $t\to\infty$, by the unique self-similar solution established in part (a) of Theorem \ref{th.uniq}. On the contrary, when the decay at infinity of the initial condition hits exactly the critical rate, the asymptotic profiles change.
\begin{theorem}\label{th.asympt.slow2}
Let $u$ be the solution to the Cauchy problem \eqref{eq1}-\eqref{ic} with an initial condition such that
\begin{equation}\label{eq.critslow}
\lim\limits_{|x|\to\infty}|x|^{(\sigma+2)/(p-m)}u_0(x)=K\in(0,\infty).
\end{equation}
Then, we have
\begin{equation}\label{asympt.slow2}
\lim\limits_{t\to\infty}t^{\alpha}|u(x,t)-U(x,t;A(K))|=0,
\end{equation}
with uniform convergence on sets of the form $\mathcal{S}_{c}$ as in \eqref{inner.sets}, where $A(K)$ is defined in part (b) of Theorem \ref{th.uniq}.
\end{theorem}
Let us remark that, while the self-similar solution $U(\cdot,\cdot;A^*)$ is an attractor for a wide class of initial conditions, as stated in Theorem \ref{th.asympt.slow1}, the self-similar solutions $U(\cdot,\cdot;A(K))$ for $K\in(0,\infty)$ attract only a rather restricted category of solutions, namely, those having initial conditions with a completely similar behavior as $|x|\to\infty$ as themselves. Another remark is that the outcome of both Theorems \ref{th.asympt.slow1} and \ref{th.asympt.slow2} does not depend on the value of $p$. For $\sigma=0$, the analogous result to Theorem \ref{th.asympt.slow2} is established in \cite{Kwak98}, while the analogous result to Theorem \ref{th.asympt.slow1} is much simpler (since the profile $f(\cdot;A^*)$ is just an explicit constant when $\sigma=0$) and corresponds to \cite[Theorem 1]{KP86}.

\medskip

\noindent \textbf{C. Asymptotic simplification for rapidly decaying initial conditions when $p>p_F(\sigma)$.} On the opposite side of the previous paragraph, we say that an initial condition $u_0$ decays rapidly at infinity if
\begin{equation}\label{dec.rapid}
\lim\limits_{|x|\to\infty}|x|^{(\sigma+2)/(p-m)}u_0(x)=0.
\end{equation}
When dealing with solutions stemming from such initial conditions, the critical exponent $p_F(\sigma)$ defined in \eqref{pcrit} plays a decisive role. In the present work, we deal only with the analysis in the range $p>p_F(\sigma)$, which involves an \emph{asymptotic simplification} and thus a convergence towards self-similar solutions of the porous medium equation that are well established. Notice first that, if $p>p_F(\sigma)$, we have
$$
\frac{\sigma+2}{p-m}<N<\infty.
$$
Moreover, let us assume that there is $\theta>(\sigma+2)/(p-m)$ such that
\begin{equation}\label{dec.fast}
\lim\limits_{|x|\to\infty}|x|^{\theta}u_0(x)=l\in(0,\infty).
\end{equation}
In order to state the following result, we recall the explicit fundamental solutions of the porous medium equation
\begin{equation}\label{PME}
\partial_tu=\Delta u^m, \quad (x,t)\in\real^N\times(0,\infty),
\end{equation}
known as \emph{Barenblatt solutions},
\begin{equation}\label{Bar.sol}
\begin{split}
&B(x,t;M)=t^{-N/(mN-N+2)}\left[D(M)-k\left(|x|t^{-1/(mN-N+2)}\right)^2\right]_+^{1/(m-1)},\\
&\quad \quad k=\frac{(m-1)N}{2N(mN-N+2)},
\end{split}
\end{equation}
the constant $D(M)>0$ being chosen such that $\|B(t;M)\|_1=M$ for any $t\in(0,\infty)$. For any $l>0$ and $0<\theta<N$, we also denote by $W_{\theta,l}$ the unique solution to \eqref{PME} with initial trace (in distributional sense) equal to $l|x|^{-\theta}$.
\begin{theorem}\label{th.asympt.fast2}
Let $p>p_F(\sigma)$ and $u$ be the solution to the Cauchy problem \eqref{eq1}-\eqref{ic} with an initial condition $u_0$ satisfying \eqref{dec.fast}. We then have
\begin{enumerate}
\item If $\theta>N$, then there is $M>0$ such that
\begin{equation}\label{asympt.fast2}
\lim\limits_{t\to\infty}t^{N/(mN-N+2)}|u(x,t)-B(x,t;M)|=0,
\end{equation}
with uniform convergence on sets of the form $\{x\in\real^N: |x|\leq ct^{1/(mN-N+2)}\}$, $c>0$.
\item If $(\sigma+2)/(p-m)<\theta<N$ and we assume furthermore that there is $C>0$ such that
\begin{equation}\label{bound.fast}
u_0(x)\leq C|x|^{-\theta}, \quad x\in\real^N,
\end{equation}
then
\begin{equation}\label{asympt.fast3}
\lim\limits_{t\to\infty}t^{\theta/(m\theta-\theta+2)}|u(x,t)-W_{\theta,l}(x,t)|=0,
\end{equation}
with uniform convergence on sets of the form $\{x\in\real^N: |x|\leq ct^{1/(m\theta-\theta+2)}\}$, $c>0$.
\end{enumerate}
\end{theorem}
Recall at this point, for the sake of completeness, that the existence and uniqueness of the solution $W_{\theta,l}$ to the porous medium equation, involved in the description of the large time behavior in \eqref{asympt.fast3}, is a classical result established in \cite{BCP84, DK84}. The previous asymptotic simplification has been established for $\sigma=0$ in \cite{KP86} and, since the limiting equation is exactly the same one, not depending on $\sigma$, the proof will have a number of technical details totally similar to the one in the above mentioned reference. This is why, we will skip some steps from the proof of Theorem \ref{th.asympt.fast2}, omitting details that are identical to the proofs in the case $\sigma=0$.

\medskip

\noindent \textbf{D. The borderline case $\theta=N$.} In order to complete the panorama of the results in the range $p>p_F(\sigma)$, we are left with the borderline case $\theta=N$, that is, to investigate the large time behavior of solutions to \eqref{eq1}-\eqref{ic} stemming from initial conditions $u_0$ satisfying
\begin{equation}\label{decay.border}
\lim\limits_{|x|\to\infty}|x|^{N}u_0(x)=l\in(0,\infty).
\end{equation}
As seen in a number of borderline cases \cite{KU87, GV91, IV09, IL17}, modified self-similar profiles involving logarithmic factors in the time scales of the evolution will appear as well in this case. More precisely, letting
$$
\eta:=\frac{1}{mN-N+2},
$$
and assuming that the initial condition satisfies the assumption
\begin{equation}\label{bound.border}
\sup\limits_{x\in\real^N}|x|^Nu_0(x)=C<\infty,
\end{equation}
we have the following result:
\begin{theorem}\label{th.asympt.border}
Let $p>p_F(\sigma)$ and let $u$ be the solution to the Cauchy problem \eqref{eq1}-\eqref{ic} with an initial condition $u_0$ satisfying \eqref{decay.border} and \eqref{bound.border}. Then there exists $M>0$ (depending on $l$) such that
\begin{equation}\label{asympt.border}
\lim\limits_{t\to\infty}t^{N\eta}\left|\frac{1}{\ln\,t}u\left(x,t(\ln\,t)^{-(m-1)}\right)-B(x,t;M)\right|=0,
\end{equation}
with uniform convergence in sets of the form
$$
\{x\in\real^N: |x|\leq ct^{\eta}\}, \quad c>0.
$$
\end{theorem}

\noindent Note that the borderline case $\theta=N$ is a kind of interpolation between the two ranges $\theta>N$ and $\theta<N$, since the time scales of the two asymptotic profiles $B(\cdot,\cdot;M)$, respectively $W_{\theta,l}$ coincide when $\theta=N$. The effect of this matching is the appearance of logarithmic corrections, as we shall see in the proof of Theorem \ref{th.asympt.border}, given at the end of the paper.

\medskip

\noindent \textbf{Initial conditions not considered in this work}. We intentionally left out of this work the large time behavior of solutions to the Cauchy problem \eqref{eq1}-\eqref{ic} with initial conditions $u_0$ satisfying \eqref{dec.rapid} (including compactly supported ones) in the range $m<p<p_F(\sigma)$, as well as in the borderline case $p=p_F(\sigma)$. Some technical difficulties of a completely different nature than the strategy performed in the proofs of the main results in this paper have been found in dealing with these cases, and they will be considered in a future work. A discussion of these difficulties and shortages is included at the end of the paper.

\section{Well posedness and universal upper bound}\label{sec.prep}

We gather in this section a number of general results in the theory of Eq. \eqref{eq1}, which have independent interest but that will also be useful later. We thus start with the well-posedness result, which is rather standard in the case of equations of absorption-diffusion. Due to its importance for the general theory of Eq. \eqref{eq1}, we state it as a theorem, but in view of precedents, its proof will be only sketched.
\begin{theorem}[Well-posedness and comparison principle]\label{th.wp}
Let $u_0\in L^{\infty}(\real^N)$ satisfying the conditions in \eqref{ic}. Then, there exists a unique weak solution to the Cauchy problem \eqref{eq1}-\eqref{ic}. Moreover, if $u_0$ and $\overline{u}_0$ are two initial conditions satisfying \eqref{ic} and such that $u_0\leq\overline{u}_0$ in $\real^N$ and $u$, $\overline{u}$ are the corresponding solutions to Eq. \eqref{eq1} with initial conditions $u_0$, $\overline{u}_0$, then
\begin{equation}\label{comp.pr}
u(x,t)\leq\overline{u}(x,t), \quad (x,t)\in\real^N\times(0,\infty).
\end{equation}
The inequality \eqref{comp.pr} remains in force if $u$ is a weak subsolution and $\overline{u}$ is a weak supersolution such that $u(x,0)\leq\overline{u}(x,0)$, $x\in\real^N$.
\end{theorem}
\begin{proof}
The \textbf{global existence} follows in a standard way by an approximation process (see for example \cite{IL25} for similar ideas). Indeed, we consider for any $\epsilon\in(0,1)$ the family of approximating equations
\begin{equation}\label{eq1.approx}
\partial_tu=\Delta u^m-\frac{|x|^{\sigma}}{1+\epsilon|x|^{\sigma}}\frac{u^p}{1+\epsilon u^{p-1}}.
\end{equation}
We observe that, for any $t\geq0$ and $\epsilon\in(0,1)$, we have
$$
\left|\frac{t}{1+\epsilon t}\right|<\frac{1}{\epsilon},
$$
and the function
$$
X\in[0,\infty)\mapsto\frac{X^p}{1+\epsilon X^{p-1}}
$$
is a Lipschitz function with a Lipschitz constant equal to $p/\epsilon$. Thus, for any $\epsilon\in(0,1)$ and $(X,Y)\in[0,\infty)^2$, we deduce that
$$
\frac{|x|^{\sigma}}{1+\epsilon|x|^{\sigma}}\left|\frac{X^p}{1+\epsilon X^{p-1}}-\frac{Y^p}{1+\epsilon Y^{p-1}}\right|
\leq\frac{p}{\epsilon^2}|X-Y|,
$$
and we find that Eq. \eqref{eq1.approx} is a Lipschitz perturbation of the porous medium equation, hence the local well-posedness of the Cauchy problem \eqref{eq1.approx}-\eqref{ic} follows by standard theory. Let $u_{\epsilon}$ be the unique solution to \eqref{eq1.approx}-\eqref{ic}. Since, for $\epsilon_1<\epsilon_2\in(0,1)$, we have
$$
\frac{|x|^{\sigma}}{1+\epsilon_1|x|^{\sigma}}\frac{X^p}{1+\epsilon_1 X^{p-1}}\geq\frac{|x|^{\sigma}}{1+\epsilon_2|x|^{\sigma}}\frac{X^p}{1+\epsilon_2 X^{p-1}}
$$
for any $x\in\real^N$ and $X\in[0,\infty)$, the comparison principle applied to Eq. \eqref{eq1.approx} entails that $u_{\epsilon_1}\geq u_{\epsilon_2}$ in $\real^N\times(0,\infty)$. Observe here that the global existence (that is, for any $t\in(0,\infty)$) of solutions is ensured by the obvious uniform upper bound
\begin{equation}\label{unif.upper}
u_{\epsilon}(x,t)\leq\|u_0\|_{\infty}, \quad (x,t)\in\real^N\times(0,\infty), \quad \epsilon\in(0,1),
\end{equation}
a bound that follows from the fact that the constant $\|u_0\|_{\infty}$ is a strict supersolution to \eqref{eq1.approx} for any $\epsilon\in(0,1)$. It then follows that there exists
$$
u(x,t):=\lim\limits_{\epsilon\to0}u_{\epsilon}(x,t), \quad (x,t)\in\real^N\times(0,\infty),
$$
and a standard argument based on passing to the limit employing the Monotone Convergence Theorem in the weak formulation of \eqref{eq1.approx} (and taking once more into account the uniform bound \eqref{unif.upper} in order to ensure its finiteness for any $t>0$) proves that $u$ is a weak solution to \eqref{eq1}, with initial condition $u_0$, completing the proof of the existence statement.

The \textbf{uniqueness and comparison principle} follow exactly as in \cite[Proof of Theorem 1.1, p. 7-8]{ILS24}, since a closer inspection of the proof therein shows that the restriction on the absorption exponent considered in the quoted reference has no influence in the proof.
\end{proof}
The next preparatory result establishes the existence of a stationary \emph{friendly giant} and a family of supersolutions of the same form.
\begin{lemma}\label{lem.fg}
Let $p>m$ and $\sigma>0$. If $m<p<p_F(\sigma)$, then the function
$$
\Gamma_{m,p,\sigma}(x)=C_*|x|^{-(\sigma+2)/(p-m)}, \quad C_*=\left[\frac{m(\sigma+2)}{p-m}\left(\frac{m\sigma+2p}{p-m}-N\right)\right]^{1/(p-m)}
$$
is a stationary solution to Eq. \eqref{eq1} in $\real^N\setminus\{0\}$ such that $\Gamma_{m,p,\sigma}\in L^1(\real^N\setminus B(0,r))$ for any $r>0$. Moreover, the function
$$
\Gamma_{C}(x)=C|x|^{-(\sigma+2)/(p-m)}
$$
is a supersolution to Eq. \eqref{eq1} for any $p>m$ and $C$ sufficiently large.
\end{lemma}
\begin{proof}
The proof follows by direct calculation. Indeed, for any $C>0$ we have
\begin{equation*}
\begin{split}
&\Delta\Gamma_{C}^m(x)=\frac{m(\sigma+2)}{p-m}C^m\left[\frac{m(\sigma+2)}{p-m}+2-N\right]|x|^{-(m\sigma+2p)/(p-m)},\\
&|x|^{\sigma}\Gamma_C^p(x)=C^p|x|^{-(m\sigma+2p)/(p-m)},
\end{split}
\end{equation*}
hence
$$
-\Delta\Gamma_C^m(x)+|x|^{\sigma}\Gamma_C^p(x)=C^m\left[C^{p-m}-\frac{m(\sigma+2)}{p-m}\left(\frac{m\sigma+2p}{p-m}-N\right)\right]
|x|^{-(m\sigma+2p)/(p-m)},
$$
which is obviously positive if $C$ is sufficiently large. In particular, if $p<N(m+\sigma)/(N-2)$ (or for any $p\in(m,\infty)$ if $N\in\{1,2\}$), the value $C=C_*$ leads to a solution, and $\Gamma_C$ is obviously a supersolution if $C>C_*$, since $p>m$. Moreover, the condition $p<p_F(\sigma)$ implies that $(\sigma+2)/(p-m)>N$ and thus $\Gamma_{m,p,\sigma}\in L^1(\real^N\setminus B(0,r))$ for any $r>0$, as claimed.
\end{proof}
Let us remark here that the above family of stationary supersolutions are useful when dealing with initial conditions satisfying \eqref{eq.critslow} or \eqref{dec.rapid}. The next result gives an \textbf{universal upper bound} for (bounded) solutions to Eq. \eqref{eq1}, proving in particular that any bounded solution must have a maximal decay rate as $|x|\to\infty$. This will be very useful when dealing with constant or slowly decaying initial conditions.
\begin{theorem}[Universal upper bound]\label{prop.decaymin}
Let $u$ be the solution to the Cauchy problem \eqref{eq1}-\eqref{ic}. Then
\begin{equation}\label{univ.bound}
u(x,t)\leq U(x,t;A^*), \quad {\rm for \ any} \ (x,t)\in\real^N\times(0,\infty),
\end{equation}
where $U(\cdot,\cdot;A^*)$ is a self-similar solution in the form \eqref{SSS} with a profile satisfying \eqref{dec.zero} as $\xi\to\infty$.
\end{theorem}
\begin{proof}
Let us observe that $U(\cdot,\cdot;A^*)$ has infinity at every point as initial trace. Indeed, fix $x\in\real^N$, $x\neq0$ (since for $x=0$ is obvious that $U(0,t;A^*)=t^{-\alpha}A^*$). We then infer from \eqref{dec.zero} that
$$
\lim\limits_{t\to0}(|x|t^{-\beta})^{\sigma/(p-1)}f(|x|t^{-\beta};A^*)=\left(\frac{1}{p-1}\right)^{1/(p-1)}.
$$
Noticing that
$$
U(x,t;A^*)=t^{-\alpha}f(|x|t^{-\beta};A^*)=t^{-1/(p-1)}|x|^{-\sigma/(p-1)}(|x|t^{-\beta})^{\sigma/(p-1)}f(|x|t^{-\beta};A^*),
$$
it readily follows that
\begin{equation}\label{interm19}
\lim\limits_{t\to0}U(x,t;A^*)=\infty, \quad x\in\real^N,
\end{equation}
with uniform convergence on compact sets. Since the initial trace of $U(\cdot,\cdot;A^*)$ is not a function in $L^{\infty}(\real^N)$, one cannot invoke directly the comparison principle in Theorem \ref{th.wp} to end the proof. This is why, we have to employ an approximation argument for comparison, as follows: given $u_0\in C(\real^N)\cap L^{\infty}(\real^N)$, for any natural number $n$, let $\{u_{0,n}\}_{n\geq1}$ be a sequence of functions such that $u_{0,n}\in C_c^{2}(B(0,n))$, $u_{0,m}\leq u_{0,n}\leq u_0$ in $B(0,m)$ for any $m<n$ natural numbers, and $u_{0,n}\to u_0$ as $n\to\infty$ uniformly on compact sets in $\real^N$. Let $\overline{u}_n$ be the (unique) solution to the Dirichlet problem
\begin{equation}\label{DP}
\left\{\begin{array}{ll}u_t=\Delta u^m-|x|^{\sigma}u^p, & {\rm in} \ B(0,n)\times(0,\infty),\\ u(x,0)=u_{0,n}(x), & {\rm for} \ x\in B(0,n), \\
u(x,t)=0, & {\rm for} \ (x,t)\in\partial B(0,n)\times(0,\infty), 
\end{array}\right.
\end{equation}
whose existence and uniqueness follow analogously as in the proof of Theorem \ref{th.wp}. The comparison principle then entails that $\overline{u}_m\leq\overline{u}_n$ in $B(0,m)$, provided $m<n$. Moreover, the uniform bound
\begin{equation}\label{UB}
\overline{u}_n(x,t)\leq \|u_0\|_{\infty}, \quad (x,t)\in\real^N\times(0,\infty), \quad n\in\mathbb{N}
\end{equation}
follows from the comparison principle, since the constant $\|u_0\|_{\infty}$ is a supersolution to the Dirichlet problem \eqref{DP}. A straightforward argument based on the dominated convergence theorem (see, for example, the proof of \cite[Proposition 2.8]{Su02} for details), together with the uniqueness part in Theorem \ref{th.wp}, ensure that the pointwise limit (which is actually uniform on compact sets)
$$
u(x,t):=\lim\limits_{n\to\infty}\overline{u}_n(x,t), \quad (x,t)\in\real^N\times(0,\infty),
$$
gives the unique solution of the Cauchy problem \eqref{eq1}-\eqref{ic} with initial condition $u_0$. Fix now a natural number $n$. The compactness of $\overline{B(0,n)}$ ensures the existence of a small time $\tau_n>0$ (depending on $n$) such that
$$
U(x,\tau;A^*)>\|u_0\|_{\infty}\geq\overline{u}_{0,n}(x), \quad x\in B(0,n), \quad \tau\in(0,\tau_n).
$$
Since now $U(\cdot,\tau;A^*)\in L^{\infty}(B(0,n))$, we can apply the comparison principle on $B(0,n)\times(0,\infty)$ to deduce that
\begin{equation*}
\overline{u}_n(x,t)\leq U(x,t+\tau;A^*), \quad (x,t)\in B(0,n)\times(0,\infty),\quad \tau\in(0,\tau_n),
\end{equation*}
and thus, by letting $\tau\to0$,
\begin{equation}\label{interm20}
\overline{u}_n(x,t)\leq U(x,t;A^*), \quad (x,t)\in B(0,n)\times(0,\infty).
\end{equation}
Since \eqref{interm20} holds true for any $n$ and thus for any ball $B(0,n)$, we can let $n\to\infty$ in \eqref{interm20} to obtain the universal bound \eqref{univ.bound}, thus completing the proof.
\end{proof}

\noindent \textbf{Remark.} The universal upper bound in Theorem \ref{prop.decaymin} is an effect of the presence of the weight $|x|^{\sigma}$. Indeed, for $\sigma=0$, there is no minimal decay as $|x|\to\infty$, and, for example, if $u_0\equiv A\in(0,\infty)$, then the solution to the Cauchy problem remains constant in $x$ at any $t>0$, namely 
$$
u(x,t)=\left((p-1)t+\frac{1}{A^{p-1}}\right)^{-1/(p-1)}.
$$

\section{Proof of Theorems \ref{th.uniq} and \ref{th.asympt.slow2}}\label{sec.uniq}

This section is dedicated to the proof of Theorem \ref{th.uniq}, which borrows ideas and techniques employed in \cite{Kwak98} and in the proof of the monotonicity and uniqueness in \cite{IM25}. We also prove Theorem \ref{th.asympt.slow2} at the end of this section, since its proof follows rather readily from the uniqueness in part (b) of Theorem \ref{th.uniq}.
\begin{proof}[Proof of Theorem \ref{th.uniq}]
(a) The existence of $A^*$ follows from \cite[Theorem 1.1]{IM25}. Assume for contradiction that there are two values $A_1<A_2\in(0,\infty)$ such that $f_i(\cdot):=f(\cdot;A_i)$, $i=1,2$ behaves as in \eqref{dec.zero} as $\xi\to\infty$. We infer from \cite[Theorem 1.1 (b)]{IM25} that both $f_1$ and $f_2$ have to be decreasing for $\xi\in(0,\infty)$ and thus \cite[Lemma 4.3]{IM25} ensures that $f_1(\xi)<f_2(\xi)$ for any $\xi\in(0,\infty)$. Similarly to \cite[Lemma 4.3]{IM25}, let us denote by $g_i=f_i^m$ and introduce the rescaling
\begin{equation}\label{resc}
f_{\lambda}(\xi)=\lambda^{-2/(m-1)}f_1(\lambda\xi), \quad g_{\lambda}(\xi)=\lambda^{-2m/(m-1)}g_1(\lambda\xi).
\end{equation}
We derive by direct calculations (see also the proof of \cite[Lemma 4.3]{IM25}) that $g_{\lambda}$ solves the differential equation
\begin{equation*}
g_{\lambda}''(\xi)+\frac{N-1}{\xi}g_{\lambda}'(\xi)+\alpha g_{\lambda}(\xi)^{1/m}+\beta\xi(g_{\lambda}^{1/m})'(\xi)-\lambda^{L/(m-1)}\xi^{\sigma}g_{\lambda}^{p/m}(\xi)=0,
\end{equation*}
with $L$ defined in \eqref{SSexp}. We deduce from \eqref{dec.zero}, \eqref{resc} and the definition of $g_i$ that
$$
\lim\limits_{\xi\to\infty}\xi^{m\sigma/(p-1)}g_{\lambda}(\xi)=K_p^m\lambda^{-mL/[(m-1)(p-1)]}>K_p^m
=\lim\limits_{\xi\to\infty}\xi^{m\sigma/(p-1)}g_{2}(\xi),
$$
for any $\lambda\in(0,1)$. It thus follows that, for any $\lambda\in(0,1)$, there is $R_{\lambda}>0$ such that $g_2(\xi)<g_{\lambda}(\xi)$ for $\xi>R_{\lambda}$. Moreover, the monotonicity of the profile $g_1$ entails that, for $\lambda'<\lambda$, we have
$$
g_{\lambda'}(\xi)=(\lambda')^{-2m/(m-1)}g_1(\lambda'\xi)>\lambda^{-2m/(m-1)}g_1(\lambda\xi)=g_{\lambda}(\xi),
$$
and consequently $R_{\lambda'}<R_{\lambda}$ if $\lambda'<\lambda$. Fixing some $\lambda\in(0,1)$, it is then easy to notice that
\begin{equation*}
\lim\limits_{\lambda'\to0}\min\limits_{[0,R_{\lambda}]}g_{\lambda'}=\lim\limits_{\lambda'\to 0}g_{\lambda'}(R_{\lambda})=\infty,
\end{equation*}
thus there exists an optimal sliding parameter
\begin{equation}\label{interm1}
\lambda_0:=\sup\{\lambda\in(0,1): g_2(\xi)<g_{\lambda}(\xi), \xi\in[0,\infty)\}\in(0,1].
\end{equation}
Observe that $\lambda_0<1$; indeed, if we assume by contradiction that $\lambda_0=1$, it follows from \eqref{interm1} that $g_{\lambda}(\xi)>g_2(\xi)$ for any $\lambda\in(0,1)$ and in particular
$$
g_{\lambda}(0)=\lambda^{-2m/(m-1)}A_1^m>A_2^m=g_2(0), \quad \lambda\in(0,1),
$$
which contradicts $A_1<A_2$. It thus follows that there exists some contact point $\xi_1\in[0,R_{\lambda_0}]$ such that $g_2(\xi_1)=g_{\lambda_0}(\xi_1)$ with $g_{\lambda_0}\geq g_2$ in a neighborhood of $\xi_1$. We proceed as in the last part of the proof of \cite[Lemma 4.3]{IM25} to deduce that such a contact point cannot exist. This contradiction proves that there cannot be two different solutions $f_1=f(\cdot;A_1)$ and $f_2=f(\cdot;A_2)$ with the same behavior \eqref{dec.zero} as $\xi\to\infty$, and thus $A^*$ is unique, as claimed.

\medskip

(b) The uniqueness part is completely similar to the proof in part (a). For $K>0$ fixed, we are left to prove the existence of $A(K)$. To this end, we observe that, if there is a profile satisfying \eqref{dec.slow} as $\xi\to\infty$, then the self-similar solution $u$ given by \eqref{SSS} with profile $f$ satisfies
\begin{equation}\label{init.trace}
\lim\limits_{t\to0}u(x,t)=K|x|^{-(\sigma+2)/(p-m)}, \quad x\in\real^N\setminus\{0\},
\end{equation}
with uniform convergence on compact sets $K\subset\real^N\setminus\{0\}$. Let us introduce the class $\mathcal{S}$ formed by solutions (not necessarily self-similar) to Eq. \eqref{eq1} having an initial trace given by \eqref{init.trace}. We next show that $\mathcal{S}$ has a minimal and a maximal element, adapting a technique from \cite{Kwak98}.

\medskip

\noindent \textbf{Minimal element of $\mathcal{S}$.} We construct the minimal element by truncation. For any natural number $j$, define
\begin{equation}\label{trunc}
v_{j,0}(x)=\min\{j,K|x|^{-(\sigma+2)/(p-m)}\}\in L^{\infty}(\real^N)
\end{equation}
and let $v_j$ be the solution to Eq. \eqref{eq1} with initial condition $v_{j,0}$ (actually, we can extend the previous definition to any $j\in(0,\infty)$, a fact that will be useful for the rescaling step at the end of the current proof). Observing that $v_{j_1,0}\leq v_{j_2,0}$ if $j_1<j_2$, we deduce from the comparison principle that $v_{j_1}\leq v_{j_2}$ provided $j_1\leq j_2$. Moreover, we deduce from Lemma \ref{lem.fg} that there is $C>0$ sufficiently large such that, for any $j\geq1$,
$$
v_{j,0}(x)\leq\Gamma_{C}(x)=C|x|^{-(\sigma+2)/(p-m)}, \quad x\in\real^N,
$$
and $\Gamma_C$ is a supersolution to Eq. \eqref{eq1}. Once more, the comparison principle gives that $v_j(x,t)\leq\Gamma_C(x)$ for any $x\in\real^N$, $t\in(0,\infty)$ and $j\geq1$. We can thus define
$$
v(x,t):=\lim\limits_{j\to\infty}v_j(x,t), \quad (x,t)\in\real^N\times(0,\infty)
$$
and observe that $v(x,t)\leq\Gamma_C(x)$ for $(x,t)\in\real^N\times(0,\infty)$ as well. Thus, $v$ is finite for $x\neq0$. Since $v_{j,0}\in L^{\infty}(\real^N)$, we also deduce from Theorem \ref{prop.decaymin} that $v_{j}(x,t)\leq U(x,t;A^*)$ and thus $v(x,t)\leq U(x,t;A^*)$ for $(x,t)\in\real^N\times(0,\infty)$, where $U(x,t;A^*)$ is the unique self-similar solution established in part (a). The latter upper bound establishes that $v$ is also finite at the origin, more precisely $v(0,t)\leq A^*t^{-\alpha}$ for any $t>0$. Since $\{v_j\}_{j\geq1}$ is an increasing sequence of solutions to Eq. \eqref{eq1}, we readily obtain by applying the Monotone Convergence Theorem in the weak formulation \eqref{weak.sol} that $v$ is a weak solution to Eq. \eqref{eq1} as well.

We next show that $v$ satisfies the condition \eqref{init.trace}. Pick $x_0\in\real^N\setminus\{0\}$ and consider the ball $B_0=B(x_0,|x_0|/2)$. Consider a function $\varphi\in C(\overline{B}_0)$ defined on the closed ball $\overline{B}_0$ such that
\begin{equation}\label{interm4}
\begin{split}
&\varphi(x)=K|x|^{-(\sigma+2)/(p-m)}, \ {\rm in} \ B(x_0,|x_0|/4), \\
&\varphi(x)=\sup\{v(x,t):(x,t)\in B_0\times(0,\infty)\}=:M, \ {\rm on} \ \partial B_0,
\end{split}
\end{equation}
and such that $\varphi(x)\geq K|x|^{-(\sigma+2)/(p-m)}$ in $B_0$. Observe that $\varphi$ is finite, since the ball $B_0$ lies at a positive distance from the origin. Let $w$ be the solution to the Dirichlet problem associated to Eq. \eqref{eq1} on $B_0$ with the conditions
\begin{equation}\label{interm5}
w(x,0)=\varphi(x), \ x\in B_0, \quad w(x,t)=M, \ x\in\partial B_0, \ t>0.
\end{equation}
We deduce from the choice of $\varphi$ that $v_j(x,0)\leq w(x,0)$ for any $j\geq1$, while the choice of $M$ gives that $v_j(x,t)\leq v(x,t)\leq w(x,t)=M$ for any $(x,t)\in\partial B_0\times(0,\infty)$. The comparison principle then gives $v_j(x,t)\leq w(x,t)$ and thus $v(x,t)\leq w(x,t)$ for any $(x,t)\in B_0\times(0,\infty)$. Taking limits as $t\to0$, it then follows that
\begin{equation}\label{interm2}
v_{j,0}(x)=v_j(x,0)\leq v(x,0)\leq w(x,0)=\varphi(x), \quad x\in B_0, \quad j\geq1
\end{equation}
and, in particular, evaluating \eqref{interm2} at $x=x_0$ and $j$ sufficiently large, we find from \eqref{trunc} and \eqref{interm4} that $v(x_0,0)=K|x_0|^{-(\sigma+2)/(p-m)}$. Since $x_0$ has been arbitrarily chosen, we have shown that $v$ satisfies \eqref{init.trace} and thus $v\in\mathcal{S}$.

Let now $w\in\mathcal{S}$. Since $w(x,0)\geq v_{j,0}(x)$ for any $x\in\real^N$ by the definition of $\mathcal{S}$ and \eqref{trunc}, the comparison principle gives $w(x,t)\geq v_j(x,t)$ and then, by letting $j\to\infty$, $w(x,t)\geq v(x,t)$ for any $(x,t)\in\real^N\times(0,\infty)$. This implies that $v$ is the minimal element of $\mathcal{S}$, as claimed.

\medskip

\noindent \textbf{Maximal element of $\mathcal{S}$.} We define
$$
V(x,t):=\sup\{v(x,t): v\in\mathcal{S}\}, \quad (x,t)\in\real^N\times(0,\infty).
$$
We want to prove that $V\in\mathcal{S}$. Since any solution to Eq. \eqref{eq1} is a subsolution to the porous medium equation, we infer from the comparison principle for the porous medium equation \eqref{PME} that, for any $v\in\mathcal{S}$, we have
$$
v(x,t)\leq W_{K}(x,t), \quad (x,t)\in\real^N\times(0,\infty),
$$
where $W_{K}$ is the unique solution to the porous medium equation with initial trace \eqref{init.trace}, its existence being a particular case of the analysis in \cite{BCP84, DK84}. Thus, $V(x,t)\leq W_K(x,t)$ for any $(x,t)\in\real^N\times(0,\infty)$, which implies the boundedness of $V(t)$ for any $t\in(0,\infty)$.

We next prove that $V$ is a weak solution to Eq. \eqref{eq1} and satisfies the initial condition \eqref{init.trace}. To this end, for any natural number $n\geq1$ we set $V_n$ to be the solution to Eq. \eqref{eq1} with initial condition $V_n(x,0)=V(x,1/n)$. By the definition of $V$, we have that
$$
u\left(x,\frac{1}{n}\right)\leq V\left(x,\frac{1}{n}\right)=V_n(x,0), \quad x\in\real^N, \quad {\rm for \ any} \ u\in\mathcal{S},
$$
and thus $u(x,t+1/n)\leq V_n(x,t)$ for any $(x,t)\in\real^N\times(0,\infty)$, $u\in\mathcal{S}$ and $n\geq1$, and passing to supremum, we have $V(x,t+1/n)\leq V_n(x,t)$ for any $t>0$. Fix now $m$, $n\geq1$ such that $m>n$. We then deduce from the previous inequality that
$$
V_n(x,0)=V\left(x,\frac{1}{n}\right)=V\left(x,\frac{1}{m}+\frac{1}{n}-\frac{1}{m}\right)\leq V_m\left(x,\frac{1}{n}-\frac{1}{m}\right), \quad x\in\real^N,
$$
hence, by the comparison principle,
\begin{equation}\label{interm3}
V_n(x,t)\leq V_m\left(x,t+\frac{1}{n}-\frac{1}{m}\right), \quad {\rm that \ is,} \quad V_n\left(x,t-\frac{1}{n}\right)\leq V_m\left(x,t-\frac{1}{m}\right),
\end{equation}
for any $m>n\geq1$ and $x\in\real^N$, $t>1/n$. We can thus define, for $(x,t)\in\real^N\times(0,\infty)$, the function
$$
W(x,t):=\lim\limits_{n\to\infty}V_n\left(x,t-\frac{1}{n}\right).
$$
The boundedness of $W$ follows as above, since $V_n$ are uniformly bounded, independent of $n$. Moreover, \eqref{interm3} and the Monotone Convergence Theorem entail that $W$ is a weak solution to Eq. \eqref{eq1}, while the fact that $V(x,t)\leq V_n(x,t-1/n)$ implies $V(x,t)\leq W(x,t)$, for any $(x,t)\in\real^N\times(0,\infty)$. We next prove that $W$ satisfies \eqref{init.trace}. To this end, pick $x_0\in\real^N\setminus\{0\}$, the ball $B_0=B(x_0,|x_0|/2)$ and the function $\varphi$ introduced in \eqref{interm4}. For any $\epsilon>0$, let $w_{\epsilon}$ be the solution to a similar Dirichlet problem as \eqref{interm5}, but replacing $\varphi$ by $(1+\epsilon)\varphi$ and $M$ by $(1+\epsilon)M$. Since \eqref{init.trace} holds true for any $v\in\mathcal{S}$, we infer from the definition of supremum that there exists $n(\epsilon)\geq1$ (depending on $\epsilon$) such that, for any $n\geq n(\epsilon)$, we have
$$
V_n(x,0)=\sup\{v(x,1/n): v\in\mathcal{S}\}\leq(1+\epsilon)K|x|^{-(\sigma+2)/(p-m)}, \quad x\in\overline{B}_0,
$$
which, together with the obvious comparison with $(1+\epsilon)M$ on $\partial B_0\times(0,\infty)$, readily implies that
$$
V_n(x,t)\leq w_{\epsilon}(x,t), \quad (x,t)\in B_0\times(0,\infty), \quad n\geq n(\epsilon).
$$
In particular, $V_n(x,t-1/n)\leq w_{\epsilon}(x,t-1/n)$ for any $x\in B_0$ and $t\geq 1/n$ and thus, by letting $n\to\infty$ and taking into account the continuity with respect to time of $w_{\epsilon}$, it follows that $W(x,t)\leq w_{\epsilon}(x,t)$, for any $(x,t)\in\real^N\times(0,\infty)$ and for any $\epsilon>0$. By evaluating the latter inequality at $x=x_0$, we find that
\begin{equation}\label{interm6}
\lim\limits_{t\to 0}W(x_0,t)\leq\lim\limits_{t\to 0}w_{\epsilon}(x_0,t)=(1+\epsilon)K|x_0|^{-(\sigma+2)/(p-m)}.
\end{equation}
Since \eqref{interm6} holds true for any $\epsilon>0$, we infer on the one hand that
\begin{equation}\label{interm7}
\lim\limits_{t\to 0}W(x_0,t)\leq K|x_0|^{-(\sigma+2)/(p-m)}.
\end{equation}
On the other hand, we have shown that $V(x,t)\leq W(x,t)$ and thus $v(x,t)\leq W(x,t)$ for any $v\in\mathcal{S}$ and $(x,t)\in\real^N\times(0,\infty)$. Passing to the limit as $t\to 0$ and taking into account \eqref{init.trace}, we readily deduce the opposite inequality to \eqref{interm7}. It then follows that \eqref{interm7} is in fact an equality and thus $W\in\mathcal{S}$. By the definition of $V$, we deduce that $W\leq V$, whence $W=V$ and thus $V\in\mathcal{S}$ and is the maximal element of $\mathcal{S}$.

\medskip

\noindent \textbf{Self-similarity and uniqueness.} Let $u$ be a solution to Eq. \eqref{eq1}. We employ an argument similar to the one in \cite[p. 2766]{IL14} and consider, for $\lambda>0$, the rescaling
\begin{equation}\label{resc.crit}
u_{\lambda}(x,t)=\lambda^{(\sigma+2)/(p-m)}u(\lambda x,\lambda^{\gamma}t), \quad \gamma=\frac{L}{p-m}.
\end{equation}
By direct calculations (whose details we omit), $u_{\lambda}$ is again a solution to Eq. \eqref{eq1}. Moreover, if $u\in\mathcal{S}$, we readily obtain that $u_{\lambda}$ also satisfies \eqref{init.trace} and thus $\mathcal{S}$ is invariant to the rescaling \eqref{resc.crit}. On the one hand, by the definition of $V$ as a supremum, $V$ has to be invariant itself to \eqref{resc.crit} and thus $V$ has self-similar form. On the other hand, recalling the construction of the minimal element $v\in\mathcal{S}$ as a limit of the approximants $v_j$ with initial condition $v_{j,0}$ given in \eqref{trunc}, we notice that
\begin{equation*}
\begin{split}
v_{\lambda,j}(x,0)&=\lambda^{(\sigma+2)/(p-m)}v_{j,0}(\lambda x)=\min\{\lambda^{(\sigma+2)/(p-m)}j,K|x|^{-(\sigma+2)/(p-m)}\}\\
&=v_{\lambda^{(\sigma+2)/(p-m)}j,0}(x).
\end{split}
\end{equation*}
It follows that $v_{\lambda,j}(x,t)=v_{\lambda^{(\sigma+2)/(p-m)}j}(x,t)$ and, passing to the limit as $j\to\infty$, we readily deduce that the minimal element $v$ is invariant to \eqref{resc.crit} and thus it also has self-similar form. Since the uniqueness of a self-similar solution satisfying \eqref{dec.slow} follows exactly as in part (a), we deduce that $v=V=U(\cdot,\cdot;A(K))$ for some $A(K)\in(0,\infty)$, completing the proof.
\end{proof}
Note that we have proved above a slightly stronger result than the one claimed in Theorem \ref{th.uniq}; that is,
\begin{corollary}\label{cor.uniq}
Let $K\in(0,\infty)$. Then, there exists $A(K)\in(0,\infty)$ such that the self-similar solution $U(\cdot,\cdot;A(K))$ is the unique solution to Eq. \eqref{eq1} with an initial trace given by \eqref{init.trace}.
\end{corollary}
This reinforced uniqueness is useful in the proof of Theorem \ref{th.asympt.slow2}, which we give below. 
\begin{proof}[Proof of Theorem \ref{th.asympt.slow2}]
Let $u$ be the solution to the Cauchy problem \eqref{eq1}-\eqref{ic} with $u_0$ satisfying \eqref{eq.critslow}. We consider the rescaling \eqref{resc.crit}, recalling that $u_{\lambda}$ is a solution to Eq. \eqref{eq1} for any $\lambda>0$. The uniform boundedness in Theorem \ref{prop.decaymin} (which is also uniform with respect to $\lambda$) and standard compactness arguments (see for example \cite{DiB83}) imply that there is a weak solution $U\in C(\real^N\times(0,\infty))$ to Eq. \eqref{eq1} and a sequence $\{\lambda_n\}_{n\geq1}$, such that $u_{\lambda_n}\to U$ with uniform convergence on compact subsets of $\real^N\times(0,\infty)$. Our goal is to identify this limit $U$ with the self-similar solution $U(\cdot,\cdot;A(K))$ satisfying \eqref{init.trace}, and this identification is done throughout the uniqueness result given in Corollary \ref{cor.uniq}. More precisely, since $u_0$ satisfies \eqref{eq.critslow}, we infer that, for any $\epsilon\in(0,K)$, there is $R(\epsilon)>0$ such that
\begin{equation}\label{interm12}
(K-\epsilon)|x|^{-(\sigma+2)/(p-m)}\leq u_0(x)\leq (K+\epsilon)|x|^{-(\sigma+2)/(p-m)}, \quad |x|\geq R(\epsilon).
\end{equation}
Fix now $x_0\in\real^N\setminus\{0\}$ and consider the ball $B(x_0,|x_0|/2)$. We obtain from \eqref{interm12} that there is $\Lambda(x_0,\epsilon)$ sufficiently large such that, for any $\lambda>\Lambda(x_0,\epsilon)$, we have
\begin{equation}\label{interm13}
(K-\epsilon)|x|^{-(\sigma+2)/(p-m)}\leq u_{\lambda,0}(x)=\lambda^{(\sigma+2)/(p-m)}u_0(\lambda x)\leq (K+\epsilon)|x|^{-(\sigma+2)/(p-m)},
\end{equation}
for any $x\in B(x_0,|x_0|/2)$. Taking into account that there is $n(x_0,\epsilon)$ sufficiently large such that $\lambda_n>\Lambda(x_0,\epsilon)$ for $n\geq n(x_0,\epsilon)$, thus \eqref{interm13} applies to $u_{\lambda_n}$, and constructing a similar function $\varphi$ as in \eqref{interm4} (but with $K$ replaced, at its turn, by $K-\epsilon$, respectively $K+\epsilon$), we conclude in the same way as in the proof of part (b) of Theorem \ref{th.uniq} that
$$
(K-\epsilon)|x_0|^{-(\sigma+2)/(p-m)}\leq\liminf\limits_{t\to0}U(x_0,t)\leq\limsup\limits_{t\to0}U(x_0,t)\leq(K+\epsilon)|x_0|^{-(\sigma+2)/(p-m)}.
$$
Since $\epsilon$ has been chose arbitrarily small, we infer that
$$
\liminf\limits_{t\to0}U(x_0,t)=K|x_0|^{-(\sigma+2)/(p-m)}, \quad x_0\in\real^N\setminus\{0\}.
$$
Theorem \ref{th.uniq} then gives that $U=U(\cdot,\cdot;A(K))$, as claimed. Recalling that $U$ is obtained as the limit of $u_{\lambda_n}$ as $n\to\infty$, we infer that in fact
$$
U(x,t;A(K))=\lim\limits_{\lambda\to\infty}u_{\lambda}(x,t),
$$
uniformly on compact sets in $\real^N\times(0,\infty)$. The proof is then completed with a standard step: letting $t=1$, we find that
\begin{equation}\label{interm11bis}
\lambda^{(\sigma+2)/(p-m)}u(\lambda x,\lambda^{L/(p-m)})\to U(x,1;A(K)), \quad {\rm as} \ \lambda\to\infty
\end{equation}
with uniform convergence on compact sets in $\real^N$. If we set $x'=\lambda x$, $t'=\lambda^{L/(p-m)}$, we get that $\lambda^{(\sigma+2)/(p-m)}=(t')^{\alpha}$ and $\lambda=(t')^{\beta}$. Thus, the convergence on compact sets in the $x$ variable in \eqref{interm11bis} translates into the convergence on sets of the form $\mathcal{S}_c$ in the $(x',t')$ variables. Dropping the primes, we arrive at \eqref{asympt.slow2} with uniform convergence on sets of the form $\mathcal{S}_c$, as stated.
\end{proof}

\section{Proof of Theorem \ref{th.asympt.slow1}}\label{sec.slow1}

We continue with the proofs of the results related to the large time behavior of solutions to Eq. \eqref{eq1} and we focus in this section on the one concerning initial conditions with the slowest decay (or no decay at all) as $|x|\to\infty$.
\begin{proof}[Proof of Theorem \ref{th.asympt.slow1}]
Let $V_K$ be the solution to Eq. \eqref{eq1} with constant initial condition $V_K(x,0)=K\in(0,\infty)$ for any $x\in\real^N$. The proof is done in two steps.

\medskip

\noindent In a \textbf{first step}, we establish the convergence \eqref{asympt.slow1} for $V_K$, for any $K\in(0,\infty)$. On the one hand, it follows from Theorem \ref{prop.decaymin} that $V_K(x,t)\leq U(x,t;A^*)$ for any $(x,t)\in\real^N\times(0,\infty)$. On the other hand, owing to the monotonicity of the profile $f(\cdot;A^*)$, we have $U(x,t;A^*)\leq A^*t^{-\alpha}$, thus there is
$$
t_{K}:=\left(\frac{K}{A^*}\right)^{-1/\alpha}\in(0,\infty)
$$
such that
$$
U(x,t_K;A^*)\leq A^*t_{K}^{-\alpha}=K=V_K(x,0), \quad x\in\real^N.
$$
The comparison principle then entails that $U(x,t+t_K;A^*)\leq V_K(x,t)\leq U(x,t;A^*)$, or equivalently
\begin{equation}\label{interm8}
(t+t_K)^{-\alpha}f(|x|(t+t_K)^{-\beta};A^*)\leq V_K(x,t)\leq t^{-\alpha}f(|x|t^{-\beta};A^*), \quad (x,t)\in\real^N\times(0,\infty).
\end{equation}
Observe first that \eqref{interm8} and the previous choice of $t_K$ readily imply that
\begin{equation}\label{sandwich}
\lim\limits_{K\to\infty}V_K(x,t)=U(x,t;A^*), \quad (x,t)\in\real^N\times(0,\infty),
\end{equation}
with uniform convergence on compact subsets of $\real^N\times(0,\infty)$. Letting then $y=xt^{-\beta}$, \eqref{interm8} writes
$$
\left(\frac{t}{t+t_K}\right)^{\alpha}f\left(|y|\left(\frac{t}{t+t_K}\right)^{\beta};A^*\right)\leq t^{\alpha}V_K(y,t)\leq f(|y|;A^*),
$$
whence
\begin{equation}\label{interm9}
\begin{split}
\left(\frac{t}{t+t_K}\right)^{\alpha}&f\left(|y|\left(\frac{t}{t+t_K}\right)^{\beta};A^*\right)-f(|y|;A^*)\leq t^{\alpha}V_K(y,t)-f(|y|;A^*)\\
&=t^{\alpha}(V_K(x,t)-U(x,t;A^*))\leq0.
\end{split}
\end{equation}
Since $f(\cdot;A^*)$ is continuous, we readily get that the left hand side of \eqref{interm9} tends to zero as $t\to\infty$ uniformly on compact sets in $\real^N$ (with respect to the $y$ variable). Recalling that $y=xt^{-\beta}$, we obtain the claimed convergence \eqref{asympt.slow1} for $V_K$, which is uniform on sets of the form $|y|\leq c$, that is $\mathcal{S}_c$.

\medskip

\noindent In a \textbf{second step}, we extend this convergence to general data $u_0$ satisfying \eqref{eq.veryslow}. To this end, let $u$ be the solution to \eqref{eq1}-\eqref{ic} with $u_0$ satisfying \eqref{eq.veryslow} and let $u_{\lambda}$ be its rescaled versions according to \eqref{resc.crit}, which are also solutions to Eq. \eqref{eq1}, having as initial condition
$$
u_{\lambda,0}(x)=\lambda^{(\sigma+2)/(p-m)}u_0(\lambda x), \quad x\in\real^N.
$$
By classical compactness arguments (see for example \cite{DiB83}), we deduce that there is a weak solution $U\in C(\real^N\times[0,\infty])$ to Eq. \eqref{eq1} and a sequence $\{\lambda_n\}_{n\geq1}$, such that $u_{\lambda_n}\to U$ uniformly on compact subsets of $\real^N\times(0,\infty)$. We next proceed as in \cite[Section 3]{KP86} and introduce, for any $K\in(0,\infty)$, the truncated initial conditions
$$
u_{\lambda,K}(x):=\min\{u_{\lambda,0}(x),K\}
$$
and the solutions $U_{\lambda,K}$ with initial conditions $u_{\lambda,K}$. On the one hand, it follows obviously from the comparison principle that $U_{\lambda,K}\leq u_{\lambda}$ in $\real^N\times(0,\infty)$, for any $\lambda>0$, $K>0$. On the other hand, we infer from \eqref{eq.veryslow} that, for any $x\in\real^N\setminus\{0\}$,
$$
\lim\limits_{\lambda\to\infty}u_{\lambda,K}(x)=\min\{K,\lim\limits_{\lambda\to\infty}\lambda^{(\sigma+2)/(p-m)}u_0(\lambda x)\}=K,
$$
thus an identical proof to the one of \cite[Lemma 3]{KP86} ensures that
\begin{equation}\label{interm10}
U_{\lambda,K}(x,t)\to V_{K}(x,t), \quad {\rm as} \ \lambda\to\infty.
\end{equation}
Since $U_{\lambda,K}\leq u_{\lambda}$ for any $\lambda>0$ and $K>0$, we deduce from \eqref{interm10} by letting $\lambda=\lambda_n$ and passing to the limit as $n\to\infty$ that
$$
V_K(x,t)\leq U(x,t)\leq U(x,t;A^*), \quad (x,t)\in\real^N\times(0,\infty),
$$
for any $K\in(0,\infty)$. We then infer from \eqref{sandwich} that $U=U(\cdot,\cdot;A^*)$ and thus $u_{\lambda}(x,t)\to U(x,t;A^*)$ as $\lambda\to\infty$, uniformly on compact sets in $\real^N\times(0,\infty)$. The proof is completed with the same standard step as at the end of the proof of Theorem \ref{th.asympt.slow2}: letting $t=1$, we find that
\begin{equation}\label{interm11}
\lambda^{(\sigma+2)/(p-m)}u(\lambda x,\lambda^{L/(p-m)})\to U(x,1;A^*), \quad {\rm as} \ \lambda\to\infty.
\end{equation}
Setting again $x'=\lambda x$, $t'=\lambda^{L/(p-m)}$, we find that $\lambda^{(\sigma+2)/(p-m)}=(t')^{\alpha}$ and $\lambda=(t')^{\beta}$ and once more the convergence on compact sets in the $x$ variable stated in \eqref{interm11} is equivalent to the convergence on sets of the form $\mathcal{S}_c$ in the $(x',t')$ variables. Removing the primes, we arrive at \eqref{asympt.slow1} with uniform convergence on sets of the form $\mathcal{S}_c$, as stated.
\end{proof}

\section{Proof of Theorem \ref{th.asympt.fast2}}\label{sec.fast}

Let us assume throughout this section that $p>p_F(\sigma)$ and that \eqref{dec.fast} is fulfilled by the initial condition $u_0$. We split the proof into two parts, corresponding to the two cases in the statement of Theorem \ref{th.asympt.fast2}.

\medskip

\noindent \textbf{Case 1: $\theta>N$.} Notice first that this choice, together with \eqref{dec.fast}, imply that $u_0\in L^1(\real^N)$. We proceed as in \cite{KP86} and introduce the scaling
\begin{equation}\label{resc.Bar}
u_{\lambda}(x,t)=\lambda^{N}u(y,s), \quad y=\lambda x, \quad s=\lambda^{N\gamma}t, \quad \lambda>0,
\end{equation}
with $\gamma>0$ to be determined. Straightforward calculation lead to the following equalities:
\begin{equation*}
\begin{split}
&\partial_tu_{\lambda}(x,t)=\lambda^{N(\gamma+1)}\partial_{s}u(y,s),\\
&\Delta u_{\lambda}^m(x,t)=\lambda^{mN+2}\Delta u^{m}(y,s),\\
&|x|^{\sigma}u_{\lambda}^p(x,t)=\lambda^{Np-\sigma}|y|^{\sigma}u^p(y,s).
\end{split}
\end{equation*}
We thus choose $\gamma$ in order to equate $\lambda^{N(\gamma+1)}=\lambda^{mN+2}$, that is,
$$
\gamma=\frac{mN-N+2}{N}
$$
and, taking into account that $u(y,s)$ is a solution to Eq. \eqref{eq1} and multiplying everything by $\lambda^{-mN-2}$, we find that $u_{\lambda}$ is a solution to
\begin{equation}\label{eq.rescN}
\partial_tu_{\lambda}=\Delta u_{\lambda}^m-\lambda^{\sigma+2+N(m-p)}|x|^{\sigma}u_{\lambda}^p, \quad (x,t)\in\real^N\times(0,\infty),
\end{equation}
together with the initial condition
\begin{equation}\label{ic.rescN}
u_{\lambda}(x,0)=\lambda^Nu_0(\lambda x).
\end{equation}
Let us consider next the solution $U_{\lambda}$ of the porous medium equation \eqref{PME} with the same initial condition as in \eqref{ic.rescN}. The comparison principle then entails that $u_{\lambda}\leq U_{\lambda}$ in $\real^N\times(0,\infty)$ and, since by well-known results in the theory of the porous medium equation (see for example \cite[Chapter 18]{VPME}),
$$
\lim\limits_{\lambda\to\infty}U_{\lambda}(x,t)=B(x,t;\|u_0\|_1),
$$
where the right hand side of the previous limit denotes the Barenblatt solution \eqref{Bar.sol} with total mass $\|u_0\|_1$, we infer from standard compactness results \cite{DiB83} that there exists a subsequence $\{\lambda_n\}_{n\geq1}$ and a function $U\in C(\real^N\times(0,\infty))$ such that $u_{\lambda_n}\to U$ as $n\to\infty$, with uniform convergence on compact subsets of $\real^N\times(0,\infty)$. Moreover, the limit satisfies
\begin{equation}\label{interm14}
U(x,t)\leq B(x,t;\|u_0\|_1), \quad (x,t)\in(\real^N\times[0,\infty))\setminus\{(0,0)\}.
\end{equation}
We insert $u_{\lambda_n}$ in the analogous weak formulation to \eqref{weak.sol} for the equation \eqref{eq.rescN} and pass to the limit as $n\to\infty$. Taking into account that the condition $p>p_F(\sigma)$ entails $\sigma+2+N(m-p)<0$, we deduce that the absorption term vanishes in the limit and thus the limit $U$ is a weak solution to the porous medium equation \eqref{PME}. Finally, by testing \eqref{eq.rescN} with a sequence of test functions approximating the constant one function, we straightforwardly deduce that
\begin{equation*}
\begin{split}
\int_{\real^N}u_{\lambda}(x,t)\,dx-\int_{\real^N}u_{\lambda}(x,0)\,dx&=-\lambda^{\sigma+2+N(m-p)}\int_0^t\int_{\real^N}|x|^{\sigma}u_{\lambda}^p(x,\tau)\,dx\,d\tau\\
&=-\lambda^{\sigma+2+N(m-p)}\int_0^t\int_{\real^N}|x|^{\sigma}\lambda^{Np}u^p(\lambda x,\lambda^{N\gamma}\tau)\,dx\,d\tau\\
&=-\lambda^{Nm+2+\sigma}\int_0^{\lambda^{N\gamma}t}\int_{\real^N}|y|^{\sigma}u^p(y,s)\lambda^{-\sigma-N-N\gamma}\,dy\,ds\\
&=-\int_0^{\lambda^{N\gamma}t}\int_{\real^N}|y|^{\sigma}u^p(y,s)\,dy\,ds.
\end{split}
\end{equation*}
We restrict ourselves to $\lambda=\lambda_n$ in the previous equality and pass to the limit as $n\to\infty$ by employing the Dominated Convergence Theorem to find
\begin{equation}\label{interm15}
\int_{\real^N}U(x,t)\,dx=\|u_0\|_1-\int_0^{\infty}\int_{\real^N}u^{p}(y,s)\,dy\,ds, \quad t>0.
\end{equation}
We infer from \eqref{interm14} and by letting $t\to 0$ in \eqref{interm15} (since the right hand side is independent of $t$) that $U(x,0)$ is a Dirac distribution concentrated at $x=0$ with total mass given by the right hand side of \eqref{interm15}. The uniqueness of the fundamental solution with a given mass to the porous medium equation \eqref{PME} and the final step of undoing the rescaling by letting $t=1$, $x'=\lambda x$ and $t'=\lambda^{N\gamma}$ similarly as in the previous sections complete the proof.

\medskip

\noindent \textbf{Case 2: $(\sigma+2)/(p-m)<\theta<N$.} Assume that $u_0$ satisfies \eqref{dec.fast} and \eqref{bound.fast}. In this case, the ideas are completely similar as in the previous one, but we notice that we can no longer employ the same rescaling \eqref{resc.Bar}, since $u_0\not\in L^{1}(\real^N)$ and thus the calculations involving $\|u_0\|_1$ are not valid. Instead, we perform a different scaling
\begin{equation*}
u_{\lambda}(x,t)=\lambda^{\theta}u(y,s), \quad y=\lambda x, \quad s=\lambda^{\gamma}t, \quad \lambda>0,
\end{equation*}
where $\theta$ is given in \eqref{dec.fast} and $\gamma>0$ is to be determined. Straightforward calculations lead to the following equalities:
\begin{equation*}
\begin{split}
&\partial_tu_{\lambda}(x,t)=\lambda^{\theta+\gamma}\partial_{s}u(y,s),\\
&\Delta u_{\lambda}^m(x,t)=\lambda^{m\theta+2}\Delta u^{m}(y,s),\\
&|x|^{\sigma}u_{\lambda}^p(x,t)=\lambda^{\theta p-\sigma}|y|^{\sigma}u^p(y,s).
\end{split}
\end{equation*}
Similarly as in Case 1, we equate the power of $\lambda$ in the first two terms in order to choose $\gamma$, that is,
\begin{equation}\label{interm16}
\gamma=(m-1)\theta+2,
\end{equation}
and after easy manipulations, we deduce that $u_{\lambda}$ solves the equation
\begin{equation}\label{eq.rescbig}
\partial_tu_{\lambda}=\Delta u_{\lambda}^m-\lambda^{\sigma+2-\theta(p-m)}|x|^{\sigma}u_{\lambda}^p, \quad (x,t)\in\real^N\times(0,\infty),
\end{equation}
together with the initial condition
\begin{equation}\label{ic.rescbig}
u_{\lambda}(x,0)=\lambda^{\theta}u_0(\lambda x).
\end{equation}
Picking $x\in\real^N\setminus\{0\}$, we observe that \eqref{dec.fast} and \eqref{ic.rescbig} give
\begin{equation}\label{interm18}
\lim\limits_{\lambda\to\infty}u_{\lambda}(x,0)=\lambda^{\theta}l|\lambda x|^{-\theta}=l|x|^{-\theta},
\end{equation}
which is an essential property in the rest of the proof. We need next the following preparatory result.
\begin{lemma}\label{lem.W}
In the previous notation and conditions, there exists $\overline{l}>0$ such that, for any $\lambda>0$,
$$
u_{\lambda}(x,t)\leq W_{\theta,\overline{l}}\left(x,t+\frac{1}{\lambda^{\gamma}}\right).
$$
\end{lemma}
\begin{proof}[Proof of Lemma \ref{lem.W}]
On the one hand, it is a well-established fact that, for any $l>0$, the solution $W_{\theta,l}$ to \eqref{PME} is in self-similar form, more precisely (see for example \cite{KP86})
\begin{equation}\label{selfsim.W}
W_{\theta,l}(x,t)=t^{-\theta/\gamma}f_l(|x|t^{-1/\gamma}),
\end{equation}
with $\gamma$ given in \eqref{interm16} and $f_l$ being a self-similar profile of the porous medium equation \eqref{PME} satisfying $f_l'(0)=0$ and
\begin{equation}\label{selfsim.prof}
\lim\limits_{\xi\to\infty}\xi^{\theta}f_l(\xi)=l.
\end{equation}
On the other hand, the condition \eqref{bound.fast} together with \eqref{selfsim.prof} and \eqref{selfsim.W} entail that there exists $\overline{l}>0$ sufficiently large such that
$$
u_0(x)\leq f_{\overline{l}}(x)=W_{\theta,\overline{l}}(x,1), \quad x\in\real^N,
$$
whence
$$
u_{\lambda,0}(x)=\lambda^{\theta}u_0(\lambda x)\leq\lambda^{\theta}W_{\theta,\overline{l}}(\lambda x,1), \quad x\in\real^N.
$$
Consider then the function
$$
\psi(x,t)=\lambda^{\theta}W_{\theta,\overline{l}}(\lambda x,\lambda^{\gamma}t+1).
$$
It follows by direct calculation that $\psi$ is a solution to the porous medium equation \eqref{PME} such that
$$
\psi(x,0)=\lambda^{\theta}W_{\theta,\overline{l}}(\lambda x,1)\geq u_{\lambda,0}(x),
$$
and by the comparison principle applied to \eqref{PME} we find that
\begin{equation*}
u_{\lambda}(x,t)\leq\psi(x,t)=\lambda^{\theta}W_{\theta,\overline{l}}(\lambda x,\lambda^{\gamma}t+1), \quad (x,t)\in\real^N\times(0,\infty).
\end{equation*}
We conclude the proof by recalling the self-similar form of $W_{\theta,\overline{l}}$ as given in \eqref{selfsim.W}, which implies that
\begin{equation*}
\begin{split}
\lambda^{\theta}W_{\theta,\overline{l}}(\lambda x,\lambda^{\gamma}t+1)&=\lambda^{\theta}(\lambda^{\gamma}t+1)^{-\theta/\gamma}f_{\overline{l}}(\lambda x(\lambda^{\gamma}t+1)^{-1/\gamma})\\
&=\left(t+\frac{1}{\lambda^{\gamma}}\right)^{-\theta/\gamma}f_{\overline{l}}\left(x\left(t+\frac{1}{\lambda^{\gamma}}\right)^{-1/\gamma}\right)\\
&=W_{\theta,\overline{l}}\left(x,t+\frac{1}{\lambda^{\gamma}}\right),
\end{split}
\end{equation*}
as stated.
\end{proof}
We can now continue with the proof of Case 2 in Theorem \ref{th.asympt.fast2}. We infer from Lemma \ref{lem.W} and the compactness result \cite{DiB83} that we can extract a subsequence $\{u_{\lambda_n}\}_{n\geq1}$ which converges as $n\to\infty$ to some function $U\in C(\real^N\times(0,\infty))$, the convergence being uniform on compact subsets of $\real^N\times(0,\infty)$. Taking into account the essential fact that, in our case,
$$
\sigma+2-\theta(p-m)<0,
$$
we infer by multiplying with test functions in the equation \eqref{eq.rescbig} and then passing to the limit as $\lambda_n\to\infty$ that $U$ is a solution to the porous medium equation \eqref{PME}. Moreover, proceeding exactly as in \cite[Lemma 4]{KP86} to establish short time estimates (as the ones therein follow only by employing the bound established in Lemma \ref{lem.W} and the properties of $W_{\theta,\overline{l}}$), we also deduce from \eqref{interm18} that
$$
\lim\limits_{t\to 0}U(x,t)=l|x|^{-\theta}, \quad x\in\real^N\setminus\{0\},
$$
in distributional sense. Thus, the uniqueness of a solution to \eqref{PME} with initial condition $l|x|^{-\theta}$ implies that $U=W_{\theta,l}$, completing the proof.

\section{Proof of Theorem \ref{th.asympt.border}}

We give here only a sketch of the proof of Theorem \ref{th.asympt.border}, since the technical details are similar to the previous proofs or to technical lemmas established for the porous medium equation in \cite{KU87}. Let us consider an initial condition $u_0$ as in \eqref{ic} and satisfying the estimates \eqref{decay.border} and \eqref{bound.border}. Observing that $u_0\not\in L^{\infty}(\real^N)$ but it is exactly at the borderline between integrability and non-integrability and that the time scales derived in \eqref{asympt.fast2} and \eqref{asympt.fast3} coincide for $\theta=N$, we perform the following scaling introducing a logarithmic factor:
\begin{equation}\label{resc.log}
u_{\lambda}(x,t)=\frac{\lambda^{N}}{\ln\,\lambda}u(y,s), \quad y=\lambda x, \quad s=\frac{\lambda^{mN-N+2}}{(\ln\,\lambda)^{m-1}}t, \quad \lambda>1.
\end{equation}
We find by direct calculations that
\begin{equation*}
\begin{split}
&\partial_tu_{\lambda}(x,t)=\frac{\lambda^{mN+2}}{(\ln\,\lambda)^{m}}\partial_su(y,s),\\
&\Delta u_{\lambda}^m(x,t)=\frac{\lambda^{mN+2}}{(\ln\,\lambda)^{m}}\Delta u^m(y,s),\\
&|x|^{\sigma}u_{\lambda}^p(x,t)=\frac{\lambda^{Np-\sigma}}{(\ln\,\lambda)^p}|y|^{\sigma}u^p(y,s),
\end{split}
\end{equation*}
and thus $u_{\lambda}$ solves the following equation
\begin{equation}\label{eq.rescborder}
\partial_tu_{\lambda}=\Delta u_{\lambda}^m-\frac{\lambda^{\sigma+2-N(p-m)}}{(\ln\,\lambda)^{m-p}}|x|^{\sigma}u_{\lambda}^p, \quad (x,t)\in\real^N\times(0,\infty),
\end{equation}
for any $\lambda>1$. Moreover, the initial condition changes in the following way:
\begin{equation}\label{init.border}
u_{\lambda}(x,0)=\frac{\lambda^N}{\ln\,\lambda}u_0(x), \quad x\in\real^N.
\end{equation}
On the one hand, we notice that the rescaling \eqref{resc.log} brings into play once again the critical exponent $\sigma+2-N(p-m)$ in the absorption term of the rescaled equation \eqref{eq.rescborder}, which is negative in our hypothesis $p>p_F(\sigma)$, which at least at a formal level suggests that an asymptotic simplification might take place in the sense that the absorption term is negligible in the limit $\lambda\to\infty$. On the other hand, the main reason for which we have considered the exact logarithmic factor $(\ln\,\lambda)^{-1}$ in \eqref{resc.log} is the fact that \eqref{init.border} and \eqref{decay.border} imply that
$$
u_{\lambda}(x,0)\to C\delta_0(x), \quad {\rm as} \ \lambda\to\infty,
$$
for some $C>0$, with convergence in distributional sense, where $\delta_0$ is a Dirac distribution concentrated at the origin. This convergence allows us to apply similar arguments as in the previous sections to establish that the limit (on subsequences) of $u_{\lambda}$ as $\lambda\to\infty$ is a solution $U$ to the porous medium equation \eqref{PME} with Dirac mass as initial trace, that is, a Barenblatt solution \eqref{Bar.sol} with the same mass. The proof of \eqref{asympt.border} is completed by undoing the rescaling \eqref{resc.log} exactly in the same way as in \cite[Theorem 1]{KU87}, which proves that the same convergence result holds true for initial conditions $u_0$ to the porous medium equation \eqref{PME} if the estimates \eqref{decay.border} and \eqref{bound.border} are fulfilled. We refrain from entering these technical details.

\section*{Discussion. Cases left out and further problems}

As mentioned in the Introduction, we have decided to refrain from considering in the present work initial conditions $u_0$ satisfying \eqref{dec.rapid} in the range $m<p<p_F(\sigma)$, as well as in the borderline case $p=p_F(\sigma)$.

In the former range, we have established in \cite{IM25} the existence of a unique very singular self-similar solution (satisfying the conditions in \eqref{VSS}) and it is rather expected that this very singular solution will represent the asymptotic profile for such solutions. However, despite the fact that existence and uniqueness of very singular solutions has been established in a number of works such as \cite{KV88, Le96, Le97, PT86}, the authors have been unable to find a proof of the analogous convergence result for $\sigma=0$ and $m>1$, the closest reference seeming to be \cite{PZ91} where such convergence towards a very singular solution is established in the supercritical fast diffusion range $m_c<m<1$. On the one hand, the initial plan of the authors to borrow and adapt ideas from \cite{IL14}, where an asymptotic convergence result to a very singular solution is proved for an equation involving the fast $p$-Laplacian diffusion and a gradient absorption term, seems to fail in a number of technical details. On the other hand, the past experience has shown that proving rigorously a convergence to a very singular solution is always a rather long process, many technical and functional-analytic properties being needed and this is why such a study would have increased by far the number of pages of the present work.

In the latter case $p=p_F(\sigma)$, note that \eqref{dec.rapid} implies $\theta>N=(\sigma+2)/(p_F(\sigma)-m)$ and thus that $u_0\in L^1(\real^N)$. An analogous approach to the one performed in the seminal paper \cite{GV91} based on a stability technique for the convergence could be expected; however, the effect of the coefficient $|x|^{\sigma}$ on the logarithmic rates is to be carefully assessed. This is why, we also refrain from entering this special case, and the two cases discussed in this section will be considered in a future work.

Studying \eqref{eq1} in the complementary range $1<p\leq m$ seems to be a very ambitious and difficult work plan. Indeed, even letting $\sigma=0$, it has been noticed that the large time behavior in these cases is rather complex, involving either a transformation leading to a generalized Fisher-KPP equation for $p=m$, whose large time behavior has been established rather recently in \cite{G20} (see also \cite{DGQ20}), or the formation of a matched asymptotic profile involving a boundary layer if $p\in(1,m)$, see \cite{CV96}. Let us mention that, to the best of our knowledge, the large time behavior to Eq. \eqref{eq1} in the range $1<p<m$ and with $\sigma=0$ is only established in dimension $N=1$. In view of these comments, one can realize the difficulty that the variable coefficient $|x|^{\sigma}$ introduces, for example when dealing with a Fisher-KPP-type equation (for $p=m$) with a weight on the absorption term, which is expected to strongly perturb the well-studied traveling wave profiles.

\bigskip

\noindent \textbf{Acknowledgements} This work is partially supported by the Spanish project PID2024-160967NB-I00.

\bigskip

\noindent \textbf{Data availability} Our manuscript has no associated data.

\bigskip

\noindent \textbf{Conflict of interest} The authors declare that there is no conflict of interest.

\bibliographystyle{plain}

\end{document}